\documentclass[10pt,a4paper,reqno]{amsart}
\usepackage{amsthm}
\usepackage{amsmath}
\usepackage{amssymb}
\usepackage{graphicx,color}

\usepackage[font=small]{caption}
\usepackage{subcaption}

\usepackage{multirow}
\usepackage[shortlabels]{enumitem}

\makeatletter
\theoremstyle{plain}
\newtheorem{thm}{Theorem}
  \theoremstyle{definition}
  \newtheorem*{thm*}{Theorem}
  
  \theoremstyle{remark}
  \newtheorem{rem}[thm]{Remark}
  \theoremstyle{plain}
  \newtheorem{prop}[thm]{Proposition}
  \theoremstyle{plain}
  \newtheorem{lem}[thm]{Lemma}
  \theoremstyle{plain}
  \newtheorem{cor}[thm]{Corollary}
 \theoremstyle{definition}
  
  \theoremstyle{remark}
  \newtheorem*{rem*}{Remark}

  \theoremstyle{definition}

\usepackage{amsfonts}
\usepackage{mathrsfs}

\newtheorem*{question*}{\it{QUESTION}}

\theoremstyle{plain}

\newcommand{\N}{\mathbb{N}}
\newcommand{\R}{{\mathbb{R}}}
\newcommand{\C}{{\mathbb{C}}}

\newcommand{\Z}{{\mathbb{Z}}}
\newcommand{\dd}{{\rm d}}
\newcommand{\ii}{{\rm i}}

\newcommand{\sign}{\mathop\mathrm{sign}\nolimits}
\renewcommand{\Re}{\mathop\mathrm{Re}\nolimits}
\renewcommand{\Im}{\mathop\mathrm{Im}\nolimits}

\def\pFq#1#2#3#4#5{ 
  {}_{#1}F_{#2}\biggl(\genfrac..{0pt}{}{#3}{#4}\biggl|\,#5\biggr)
}

\makeatother

\begin{document}

\title[]{Asymptotic spectral properties of the Hilbert $L$-matrix}

\dedicatory{Dedicated to the memory of Harold Widom.}

\author{Franti{\v s}ek {\v S}tampach}
\address[Franti{\v s}ek {\v S}tampach]{
	Department of Mathematics, Faculty of Nuclear Sciences and Physical Engineering, Czech Technical University in Prague, Trojanova~13, 12000 Praha~2, Czech Republic
	}	
\email{stampfra@fjfi.cvut.cz}

\subjclass[2020]{15B99, 15A18, 47B36, 47B35}

\keywords{Hilbert $L$-matrix, $L$-matrices, asymptotic spectral analysis}

\date{\today}

\begin{abstract}
We study asymptotic spectral properties of the generalized Hilbert $L$-matrix
\[
 L_{n}(\nu)=\left(\frac{1}{\max(i,j)+\nu}\right)_{i,j=0}^{n-1},
\]
for large order $n$. First, for general $\nu\neq0,-1,-2,\dots$, we deduce the asymptotic distribution of eigenvalues of $L_{n}(\nu)$ outside the origin. Second, for $\nu>0$, asymptotic formulas for small eigenvalues of $L_{n}(\nu)$ are derived. Third, in the classical case $\nu=1$, we also prove asymptotic formulas for large eigenvalues of $L_{n}\equiv L_{n}(1)$. In particular, we obtain an asymptotic expansion of $\|L_{n}\|$ improving Wilf's formula for the best constant in truncated Hardy's inequality.
\end{abstract}

\maketitle

\section{Introduction}

The (generalized) Hilbert $L$-matrix
\[
L(\nu):=\begingroup 
 \setlength\arraycolsep{2pt}\def\arraystretch{1.2} 
 \begin{pmatrix}
 \frac{1}{\nu} & \frac{1}{\nu+1} & \frac{1}{\nu+2} & \dots \\
 \frac{1}{\nu+1} & \frac{1}{\nu+1} & \frac{1}{\nu+2} & \dots \\
 \frac{1}{\nu+2} & \frac{1}{\nu+2} & \frac{1}{\nu+2} & \dots \\
 \vdots & \vdots & \vdots & \ddots 
 \end{pmatrix},
 \endgroup
\]
raised a recent interest for its peculiar properties and similarity to the famous (generalized) Hilbert matrix 
\[
H(\nu):=\begingroup 
 \setlength\arraycolsep{2pt}\def\arraystretch{1.2} 
 \begin{pmatrix}
 \frac{1}{\nu} & \frac{1}{\nu+1} & \frac{1}{\nu+2} & \dots \\
 \frac{1}{\nu+1} & \frac{1}{\nu+2} & \frac{1}{\nu+3} & \dots \\
 \frac{1}{\nu+2} & \frac{1}{\nu+3} & \frac{1}{\nu+4} & \dots \\
 \vdots & \vdots & \vdots & \ddots 
 \end{pmatrix},
 \endgroup
\]
see~\cite{bou-mas_oam21a,bou-mas_laa22,sta_jfa22}.
For $n\in\N$, we denote by $L_{n}(\nu)$ and $H_{n}(\nu)$ the $n\times n$ sections of $L(\nu)$ and $H(\nu)$, i.e., 
\[
 \left(L_{n}(\nu)\right)_{i,j}=\frac{1}{\max(i,j)+\nu} \quad\mbox{ and }\quad \left(H_{n}(\nu)\right)_{i,j}=\frac{1}{i+j+\nu},
\]
for $i,j=0,1,\dots,n-1$. Further, if $\nu=1$, we write simply $L_{n}\equiv L_{n}(1)$ and $H_{n}\equiv H_{n}(1)$. The Hilbert $L$-matrix appears less frequently than its Hankel counterpart but $L_{n}$ is used, for example, in Choi's tricks or treats~\cite{choi_amm83} and referred to as \emph{the loyal companion} of the classical Hilbert matrix $H_{n}$. In fact, $L_{n}$ appears even earlier in work of Wilf~\cite{wil_70}, who recognized in $\|L_{n}\|$ the optimal constant in the truncated classical Hardy inequality; see~\eqref{eq:disc_hardy_ineq} below.

While generalized Hilbert matrix $H(\nu)$ is the prominent example of a Hankel matrix its loyal companion $L(\nu)$ is a less known example of the class of matrices whose entries are arranged in the reversed $L$-shape and therefore briefly called $L$-matrices~\cite{bou-mas_oam21a,bou-mas_oam21b}. The norm of semi-infinite $L$-matrices, regarded as operators on $\ell^{2}(\N_{0})$, was investigated by Bouthat and Masherghi in~\cite{bou-mas_oam21a}. In particular, the authors observed that $\|L(\nu)\|=4$ for all $\nu\geq1/2$. In fact, $\|L(\nu)\|=4$ for all $\nu\geq\nu_{0}\approx0.3491$, where the threshold value $\nu_{0}$ can be identified as a~unique positive zero of an ${}_{3}F_{2}$-hypergeometric function, see~\cite[Thm.~29]{sta_jfa22} for more details. Explicit lower and upper bounds on $\nu_{0}$ were obtained only recently in~\cite{bou-mas_laa22}. Concerning spectral properties of $L(\nu)$, the spectrum of $L(\nu)$ is purely absolutely continuous and fills the interval $[0,4]$ for $\nu\geq\nu_{0}$. If $0<\nu<\nu_{0}$, a unique eigenvalue greater than $4$ appears in the spectrum of $L(\nu)$. For more information on spectral properties of $L(\nu)$, we refer the reader  to~\cite{sta_jfa22}.

This paper aims to establish asymptotic spectral properties of matrix $L_{n}(\nu)$ for $n\to\infty$. As a~motivation and also for comparison reasons, we first recall known results for the Hilbert matrix. We denote by $\lambda_{1,n}(\nu)\leq\lambda_{2,n}(\nu)\leq\dots\leq\lambda_{n,n}(\nu)$ eigenvalues of $H_{n}(\nu)$. 

From the numerical point of view, $H_{n}(\nu)$ is a canonical example of ill-conditioned Hankel matrix~\cite{todd_nbsams54,todd_jrnbs61}. For $\nu>0$, matrix $H_{n}(\nu)$ is positive-definite for all $n\in\N$ and, if $\nu\geq1/2$, the spectrum of $H(\nu)$ is purely absolutely continuous filling the interval $[0,\pi]$, see~\cite{ros_pams58} or~\cite[Thm.~8]{kal-sto_lma16}. It follows that $0<\lambda_{1,n}(\nu)\leq\dots\leq\lambda_{n,n}(\nu)\leq\pi$, for $\nu\geq1/2$, and the eigenvalues fill densely the interval $[0,\pi]$, as $n\to\infty$. However, they are distributed very irregularly. For $n$ large, most eigenvalues are located in a right neighborhood of $0$ while they appear only sporadically in a left neighborhood of $\pi$. In fact, Widom showed that the number of the eigenvalues located outside a neighborhood of $0$ is proportional to $\log n$ for $n$ large. More precisely, in~\cite{wid_tams66}, Widom proved the asymptotic formula, which was previously conjectured by Wilf in~\cite{wil_jmaa64}, and which implies that, for any $x\in(0,1)$ and $\nu\in\R\setminus(-\N_{0})$, we have
\begin{equation}
\lim_{n\to\infty}\frac{\#\{\lambda\in\sigma(H_{n}(\nu)) \mid \pi x<\lambda<\pi \}}{\log n}=\frac{1}{\pi^{2}}\log\left(\frac{1+\sqrt{1-x^{2}}}{x}\right),
\label{eq:hilbert_number_evls_limit}
\end{equation}
where $\# S$ is the cardinality of the set $S$ (there is a missing factor $1/(2\pi)$ in~\cite{wid_tams66}).

A closely related problem of interest aims to deduce an asymptotic behavior of the extreme eigenvalues $\lambda_{1,n}(\nu)$ and $\lambda_{n,n}(\nu)$, for $n\to\infty$. In~\cite{wid-wil_pams66}, Widom and Wilf showed that the smallest eigenvalue of a Hankel matrix, whose entries are moments of a compactly supported measure from the Szeg\H{o} class, exhibits an exponential decay; see also~\cite{ber-szw_ca11} for more recent developments. As an application, the authors give an asymptotic formula for $\lambda_{1,n}(1)$ in the end of their paper (the formula contains a typo which was corrected in~\cite{wid-wil_pams68}). In a greater generality, for $\nu>0$, we have
\begin{equation}
 \lambda_{1,n}(\nu)=\frac{2^{15/4}\pi^{3/2}}{(1+\sqrt{2})^{2\nu-2}}\frac{\sqrt{n}}{(1+\sqrt{2})^{4n}}\left(1+o(1)\right), \quad n\to\infty,
\label{eq:hilbert_smallest_evl_asympt}
\end{equation}
see~\cite{kal_faa01}. As far as the largest eigenvalue $\lambda_{n,n}(\nu)$, which coincides with $\|H_{n}(\nu)\|$ for $\nu>0$, is concerned, we have to mention the spectacular asymptotic formula for $\lambda_{n,n}(2)$ due to de~Bruijn and Wilf~\cite{deb-wil_bams62}:
\begin{equation}
 \lambda_{n,n}(2)=\pi-\frac{\pi^{5}}{2\log^{2} n}+O\!\left(\frac{\log\log n}{\log^{3}n}\right), \quad n\to\infty.
\label{eq:hilbert_largest_evl_asympt} 
\end{equation}
Number $\lambda_{n,n}(2)$ is the optimal constant in truncated Hilbert's inequality
\[
 \sum_{i,j=1}^{n}\frac{x_{i}x_{j}}{i+j}\leq\lambda_{n,n}(2)\sum_{i=1}^{n}x_{i}^{2},
\]
which holds true for all $x_{1},\dots,x_{n}\in\R$. Let us also remark that perhaps the best up to date known upper bound on $\|H_{n}\|$ was obtained by Otte in~\cite{ott_pjm05}.

To prove~\eqref{eq:hilbert_largest_evl_asympt}, de~Bruijn and Wilf relate matrix $H_{n}(2)$ to an integral operator and elaborate on ideas of Widom~\cite{wid_tams58,wid-tams61}, see~\cite[Sec.~2.5]{wil_70}. In fact, Wilf realized that the method admits a generalization and obtained the following nice result, see~\cite{wil_im62} or~\cite[Sec.~2.6]{wil_70}.

\begin{thm*}[Wilf]
\emph{
 Let $\mathcal{K}:(0,\infty)\times(0,\infty)\to\R$ be symmetric, homogeneous of degree $-1$, and decreasing kernel such that
 \[
  \mathcal{K}(x,1)=O\left(x^{-1/2-\delta}\right), \quad x\to\infty,
 \]
 for some $\delta>0$. Then the norm of matrix $K_{n}:=\left(\mathcal{K}(i,j)\right)_{i,j=1}^{n}$ satisfies
 \[
  \|K_{n}\|=A-\frac{B\pi^{2}}{\log^{2}n}+O\left(\frac{\log\log n}{\log^{3} n}\right), \quad n\to\infty,
 \]
 where
 \[
  A=\int_{0}^{\infty}\frac{\mathcal{K}(x,1)}{\sqrt{x}}\,\dd x \quad\mbox{ and }\quad 
  B=\int_{1}^{\infty}\frac{\log^{2}x}{\sqrt{x}}\,\mathcal{K}(x,1)\,\dd x.
 \]
 }
\end{thm*}

Since the kernel $\mathcal{K}(x,y)=1/\max(x,y)$ fulfills the assumptions (in particular is homogeneous of degree $-1$) a straightforward application of Wilf's theorem yields an asymptotic formula for $\|L_{n}\|$ in the form~\cite[Eq.~2.67]{wil_70}
\begin{equation}
 \|L_{n}\|= 4-\frac{16\pi^{2}}{\log^{2}n}+O\left(\frac{\log\log n}{\log^{3} n}\right), \quad n\to\infty.
\label{eq:norm_L_n_wilf}
\end{equation}
Number $\|L_{n}\|$ is of particular interest since it coincides with the optimal constant in truncated Hardy's inequality
\begin{equation}
 \sum_{k=1}^{n}\left(\frac{x_{1}+\dots +x_{k}}{k}\right)^{\!2}\leq\|L_{n}\|\sum_{k=1}^{n}x_{k}^{2},
\label{eq:disc_hardy_ineq}
\end{equation}
where $x_{1},\dots,x_{n}\in\R$. Indeed, the matrix from the quadratic form on the left-hand side in~\eqref{eq:disc_hardy_ineq}, say $M_{n}$, whose entries read
\[
 \left(M_{n}\right)_{i,j}=\sum_{l=\max(i,j)}^{n}\frac{1}{l^{2}},
\]
for $i,j\in\{1,\dots,n\}$, is similar to $L_{n}$ via the similarity transform $L_{n}=C_{n}M_{n}C_{n}^{-1}$, where
\[
\left(C_{n}\right)_{i,j}=
\begin{cases}
 \frac{1}{i},& \quad \mbox{ if } 1\leq j\leq i\leq n,\\
 0,& \quad \mbox{ otherwise, }
\end{cases}
\]
is the Ces\`{a}ro matrix. Thus, $M_{n}$ and $L_{n}$ have the same eigenvalues. Of course, \eqref{eq:hilbert_largest_evl_asympt} can be reproduced by a similar application of Wilf's theorem to the kernel $\mathcal{K}(x,y)=1/(x+y)$ (the homogeneity of $\mathcal{K}$ of order $-1$ is the reason for $\nu=2$ in~\eqref{eq:hilbert_largest_evl_asympt}).

The goal of this article is to deduce a similar result to~\eqref{eq:hilbert_number_evls_limit} in the case of matrix $L_{n}(\nu)$ for $\nu\in\R\setminus(-\N_{0})$, analyze asymptotic behavior of small eigenvalues of $L_{n}(\nu)$ for $\nu>0$, and also large eigenvalues but only for $\nu=1$. In particular, we improve Wilf's asymptotic formula~\eqref{eq:norm_L_n_wilf}. Let us denote by $\mu_{1,n}(\nu)\leq\mu_{2,n}(\nu)\leq\dots\leq\mu_{n,n}(\nu)$ eigenvalues of $L_{n}(\nu)$ (in fact, the eigenvalues are all distinct as shown in Theorem~\ref{thm:spec_L_n} below). If $\nu=1$, we write $\mu_{j,n}\equiv\mu_{j,n}(1)$ for all $j\in\{1,\dots,n\}$. Our main results are summarized in the following 3 theorems.

\begin{thm}\label{thm:l-hilbert_limit}
 For $\nu\in\R\setminus(-\N_{0})$ and $x\in(0,1)$, we have 
 \[
  \lim_{n\to\infty}\frac{\#\{\mu\in\sigma(L_{n}(\nu)) \mid 4x<\mu<4\}}{\log n}=\frac{1}{2\pi}\sqrt{\frac{1-x}{x}}.
 \]
\end{thm}

\begin{thm}\label{thm:l-hilbert_small_evls}
 For $\nu>0$ and $j\in\N$ fixed, we have
 \[
 \mu_{j,n}(\nu)=\frac{1}{4n^{2}}\left[1+\frac{i_{j}}{\sqrt[3]{3}}\,n^{-2/3}+o\left(n^{-2/3}\right)\right]\!,\quad n\to\infty,
 \]
 where $i_{1}<i_{2}<\dots$ are positive zeros of the Airy function, see~\eqref{eq:airy}.
\end{thm}

\begin{thm}\label{thm:l-hilbert_large_evls}
 For $j\in\N$ fixed, we have
 \[
 \mu_{n-j+1,n}= 4-\frac{16\pi^{2}j^{2}}{\log^{2}n}+\frac{32\pi^{2}j^{2}\left(\gamma+6\log 2\right)}{\log^{3}n}+O\left(\frac{1}{\log^{4}n}\right), \quad n\to\infty,
 \]
 where $\gamma$ is the Euler--Mascheroni constant. In particular, the asymptotic expansion for the norm of $L_{n}$ reads
 \[
 \|L_{n}\|= 4-\frac{16\pi^{2}}{\log^{2}n}+\frac{32\pi^{2}\left(\gamma+6\log 2\right)}{\log^{3}n}+O\left(\frac{1}{\log^{4}n}\right), \quad n\to\infty.
 \]  
 (See also~\eqref{eq:norm_asympt} for an expansion in higher order terms.)
\end{thm}

Let us briefly compare the obtained results for the Hilbert $L$-matrix with the known asymptotic properties of the Hilbert matrix. The stable part of the spectrum of $H(\nu)$ is the interval $[0,\pi]$ which, in the case of $L(\nu)$, is replaced by the interval $[0,4]$. We see from Theorem~\ref{thm:l-hilbert_limit} that the number of eigenvalues of $L_{n}(\nu)$ outside any neighborhood of $0$ is again proportional to $\log n$, for $n$ large, as in the case of $H_{n}(\nu)$, although the limiting distribution is dictated by a different function, recall~\eqref{eq:hilbert_number_evls_limit}. Notice different orders of the singularities at the origin of the distribution functions which are $O(\log x)$ and $O(1/\sqrt{x})$, as $x\to0+$, in the case of~$H_{n}(\nu)$ and $L_{n}(\nu)$, respectively. 

By comparing Theorem~\ref{thm:l-hilbert_small_evls} with~\eqref{eq:hilbert_smallest_evl_asympt}, we see that, for $\nu>0$, the smallest eigenvalue of $L_{n}(\nu)$ decays as $1/n^{2}$ rather than exponentially for $n$ large. Notice also that the first two terms in the asymptotic expansion of $\mu_{j,n}(\nu)$, for $n\to\infty$, are independent of parameter $\nu>0$.

By Theorem~\ref{thm:l-hilbert_large_evls} and~\eqref{eq:hilbert_largest_evl_asympt}, asymptotic behaviors of the largest eigenvalues of $H_{n}(2)$ and $L_{n}$ are comparable. Both tend to the respective limit points, either $\pi$ or $4$ at the rate proportional to $1/\log^{2} n$. Moreover, we see from Theorem~\ref{thm:l-hilbert_large_evls} that the remainder in Wilf's formula~\eqref{eq:norm_L_n_wilf} is actually $O(1/\log^{3} n)$. In fact, it follows from our proof of Theorem~\ref{thm:l-hilbert_large_evls} that $\|L_{n}\|$ can be asymptotically expanded to a linear combination of negative powers of $\log n$ up to an arbitrary order, however, higher order coefficients become quickly complicated, see~\eqref{eq:norm_asympt}.

Let us remark that we do not try to investigate an asymptotic behavior of large eigenvalues of $L_{n}(\nu)$ for more general values of parameter $\nu$ in this article. Our proof differs from Wilf's approach. Rather than approximating by integral operators, we analyze asymptotic behavior of the characteristic polynomial of $L_{n}$ directly. Although the restriction to homogeneous kernels of order $-1$ is of no importance in the present method, an extraction of the leading term in an asymptotic formula for the characteristic polynomial of $L_{n}$ has to be accompanied by a certain explicit control of the remainder, see Proposition~\ref{prop:q_n_asympt}. It is possible for $\nu=1$ since formulas are still of a reasonable complexity. In a greater generality, if $\nu\geq\nu_{0}$, the approach seems to be still applicable since respective formulas remain expressible as integrals of hypergeometric functions, but their analysis will be more involved. Some flavor of the formulas can be seen from~\cite{sta_jfa22}. If $0<\nu<\nu_{0}$, an asymptotic expansion of $\|L_{n}(\nu)\|$ is expected to be less explicit since the norm of the limiting operator $L(\nu)$ is known only implicitly as a~zero of a~special function. A similar remark applies also to small eigenvalues of $L_{n}(\nu)$, for $\nu<0$, because a negative eigenvalue appears in the spectrum of $L(\nu)$ and the Mat{\' e}--Nevai--Totik theorem used in the proof of Theorem~\ref{thm:l-hilbert_small_evls} is no more applicable.

We want to mention a nice and relevant result due to Pushnitski that appeared only recently in~\cite[Thm.~2.1]{pus_preprint}. Similarly as the above mentioned Wilf's theorem, it applies to a general class of matrices given by symmetric homogeneous of degree $-1$ kernels and gives their asymptotic spectral densities. In particular, it readily yields formula~\eqref{eq:hilbert_number_evls_limit} for $\nu=2$ as well as Theorem~\ref{thm:l-hilbert_limit} for $\nu=1$; see~\cite[Examples~3.1 and~3.2]{pus_preprint}.

The paper is organized as follows. In Section~\ref{sec:char_pol}, preliminary results on the characteristic polynomial and the spectrum of $L_{n}(\nu)$ are deduced. Particularly, the characteristic polynomial of $L_{n}$ is recognized as a terminating hypergeometric series which is used later in the proof of Theorem~\ref{thm:l-hilbert_large_evls}. The proof of Theorem~\ref{thm:l-hilbert_limit} is given in Section~\ref{sec:number_evls_limit}. The asymptotic behavior of small eigenvalues of $L_{n}(\nu)$, for $\nu>0$, is explored in Section~\ref{sec:small_evls}. An asymptotic analysis of large eigenvalues of $L_{n}$ is worked out in Section~\ref{sec:large_evls}. The paper is concluded by an appendix, where a proof of a technical lemma is postponed to and few more computational details of higher order coefficients in the asymptotic expansion of $\|L_{n}\|$ are given.

\section{Characteristic polynomial and spectrum}\label{sec:char_pol}

An important observation for an upcoming analysis is that $L_{n}(\nu)$ is a regular matrix with a tridiagonal inverse. More generally, by a simple linear algebra, one checks that the determinant of an $L$-matrix
\begin{equation}
 A_{n}=
 \begin{pmatrix}
 a_{0} & a_{1} & a_{2} & \dots & a_{n-1}\\
 a_{1} & a_{1} & a_{2} & \dots & a_{n-1}\\
 a_{2} & a_{2} & a_{2} & \dots & a_{n-1}\\
 \vdots & \vdots & \vdots & \ddots & \vdots \\
 a_{n-1} & a_{n-1} & a_{n-1} & \dots & a_{n-1}\\
 \end{pmatrix}
\label{eq:def_L-matrix}
\end{equation}
reads
\begin{equation}
 \det A_{n}=a_{n-1}\prod_{j=1}^{n-1}(a_{j-1}-a_{j}).
\label{eq:det_A}
\end{equation}
Consequently, $A_{n}$ is regular if and only if $a_{j-1}\neq a_{j}$, for all $j\in\{1,2,\dots,n-1\}$, and $a_{n-1}\neq0$. In this case, the inverse of $A_{n}$ is the tridiagonal matrix 

\begin{equation}
A_{n}^{-1}=
\begin{pmatrix}
 b_{0} & -b_{0} \\
 -b_{0} & b_{0}+b_{1} & -b_{1} \\
 & -b_{1} & b_{1}+b_{2} & -b_{2} \\
 & & \ddots & \ddots & \ddots \\
 & & & -b_{n-3} & b_{n-3}+b_{n-2} & -b_{n-2} \\
 & & & & -b_{n-2} & \frac{a_{n-2}}{a_{n-1}}b_{n-2}
 \end{pmatrix},
\label{eq:matrix_A_inv}
\end{equation}
where
\[
b_{j}=\frac{1}{a_{j}-a_{j+1}},
\]
which can be verified by a direct multiplication of the two matrices. For later purposes, we also define the Jacobi matrix
\begin{equation}
B_{n}=
\begin{pmatrix}
 b_{0} & -b_{0} \\
 -b_{0} & b_{0}+b_{1} & -b_{1} \\
 & -b_{1} & b_{1}+b_{2} & -b_{2} \\
 & & \ddots & \ddots & \ddots \\
 & & & -b_{n-3} & b_{n-3}+b_{n-2} & -b_{n-2} \\
 & & & & -b_{n-2} & b_{n-2}+b_{n-1}
 \end{pmatrix},
\label{eq:matrix_B}
\end{equation}
which differs from~\eqref{eq:matrix_A_inv} only by the last diagonal entry
\begin{equation}
 A_{n}^{-1}=B_{n}-\frac{a_{n}b_{n-1}}{a_{n-1}}e_{n}e_{n}^{T}.
\label{eq:A_rel_B_rank-one}
\end{equation}
Here $e_{n}=(0,0,\dots,0,1)^{T}$ denotes the last vector of the standard basis of $\C^{n}$. This fact implies a relation between characteristic polynomials of $A_{n}$ and $B_{n}$ that can be expressed as follows:
\begin{equation}
  \det(1-zA_{n})=(-1)^{n}\left[\prod_{j=1}^{n-1}(a_{j-1}-a_{j})\right]\left[a_{n-1}\det(z-B_{n})+a_{n}b_{n-1}\det(z-B_{n-1})\right],
\label{eq:char_pol_A_n}
\end{equation}
for all $z\in\C$ and $n\geq2$; see the proof of \cite[Thm.~13]{sta_jfa22} for details.

The generalized Hilbert $L$-matrix coincides with~\eqref{eq:def_L-matrix} for
\[
 a_{j}:=\frac{1}{j+\nu}.
\]
Then entries of $B_{n}$ are given by the sequence
\[
 b_{j}=\frac{1}{a_{j}-a_{j+1}}=(j+\nu)(j+1+\nu),
\]
and equation~\eqref{eq:char_pol_A_n} can be written as
\begin{equation}
\det(1-z L_{n}(\nu))=(n+\nu)P_{n}(z;\nu)-(n+\nu-1)P_{n-1}(z;\nu),
\label{eq:char_pol_L_n_P_n}
\end{equation}
for $n\geq2$, where
\[
 P_{n}(z;\nu):=\frac{(-1)^{n}}{(\nu)_{n}(\nu+1)_{n}}\det(z-B_{n})
\]
are orthonormal polynomials corresponding to Jacobi matrix~$B_{n}$ and $(\alpha)_{n}:=\alpha(\alpha+1)\dots(\alpha+n-1)$ is the Pochhammer symbol (the measure of orthogonality of polynomials $P_{n}(\;\cdot\;;\nu)$ was obtained in~\cite[Prop.~26]{sta_jfa22}). It was shown in~\cite[Eq.~38]{sta_jfa22} that polynomials $P_{n}$ can be expressed in terms of the hypergeometric series, namely
\begin{equation}
 P_{n}\!\left(\frac{1}{4}-z^{2};\nu\right)=\frac{\pi}{\sin(2\pi z)}\left[\chi(z;\nu)\phi_{n}(-z;\nu)-\chi(-z;\nu)\phi_{n}(z;\nu)\right],
\label{eq:P_rel_chi_phi}
\end{equation}
where
\begin{equation}
\phi_{n}(z;\nu)=\frac{\Gamma(n+\nu+1)}{\Gamma(2z+1)\Gamma(n+\nu+z+3/2)}\,\pFq{3}{2}{z+1/2, z+1/2, z+3/2}{2z+1,n+\nu+z+3/2}{1},
\label{eq:def_phi}
\end{equation}
for $n\in\N_{0}$, $n+\nu>0$, $z\in\C$, and
\begin{equation}
   \chi(z;\nu)=\frac{(z+1/2)\Gamma(\nu)\Gamma(\nu+1)}{\Gamma(z+1/2)\Gamma^{2}(z+\nu+1/2)}\,\pFq{3}{2}{\nu-1,\nu+1,z+1/2}{z+\nu+1/2,z+\nu+1/2}{1},
\label{eq:def_chi}
\end{equation}
for $\nu\in\R\setminus(-\N_{0})$ and $\Re z>-1/2$. It is explained in~\cite[Subsec.~3.2]{sta_jfa22} that both functions $\phi_{n}$ and $\chi$ extend analytically to entire functions in $z$ and meromorphic functions in $\nu$ with poles in $-\N_{0}$. A combination of formulas~\eqref{eq:char_pol_L_n_P_n} and~\eqref{eq:P_rel_chi_phi} yields a somewhat complicated expression for the characteristic polynomial of $L_{n}(\nu)$, however, it is useful for an asymptotic analysis which will be demonstrated in Section~\ref{sec:number_evls_limit}.

Next, we summarize basic properties of eigenvalues of $L_{n}(\nu)$. Recall that $\|L(\nu)\|=4$, if $\nu\geq\nu_{0}\approx0.3491$, and $\|L(\nu)\|\to\infty$, for $\nu\to0+$; see~\cite[Thms.~37 and~38]{sta_jfa22}.

\begin{thm}\label{thm:spec_L_n}
For all $n\in\N$ and $\nu\in\R\setminus(-\N_{0})$, eigenvalues of $L_{n}(\nu)$ are simple and, if $\nu>0$, then 
\[
 0<\mu_{1,n}(\nu)<\mu_{2,n}(\nu)<\dots<\mu_{n,n}(\nu)\leq\|L(\nu)\|.
\]
\end{thm}

\begin{proof}
Let $n\in\N$ and $-\nu\notin\N_{0}$. Since $L_{n}^{-1}(\nu)$ is a tridiagonal matrix with non-vanishing off-diagonal elements its eigenvectors are determined uniquely by the eigenvalue equation up to multiplicative constants. Hence all eigenvalues of $L_{n}^{-1}(\nu)$ are simple and the same is then true for $L_{n}(\nu)$, of course.

It follows from the general formula for determinant~\eqref{eq:det_A} that an $L$-matrix $A_{n}$ is positive definite if and only if $a_{0}>a_{1}>\dots>a_{n-1}>0$, which is true for the case of matrix $L_{n}(\nu)$, where $a_{n}=1/(n+\nu)$, if $\nu>0$. This means that $\mu_{1,n}(\nu)>0$ for all $\nu>0$ and $n\in\N$.

Finally, assuming still that $\nu>0$, we have
\[
 \mu_{n,n}(\nu)=\|L_{n}(\nu)\|=\|P_{n}L(\nu)P_{n}\|\leq\|L(\nu)\|,
\]
where we denote by $P_{n}$ the orthogonal projection onto the span of the first $n$ vectors of the standard basis of~$\ell^{2}(\N_{0})$.
\end{proof}

Recall that we do not designate the dependence on $\nu$ when $\nu=1$, i.e., $L_{n}\equiv L_{n}(1)$ and $\mu_{j,n}\equiv \mu_{j,n}(1)$. In Section~\ref{sec:large_evls}, we restrict ourselves to the particular case $\nu=1$. As preliminary results, we show here the formula for $\det(1-zL_{n})$ simplifies to a form of a single terminating hypergeometric series and summarize basic properties of the spectrum of $L_{n}$ in the next statement.

\begin{thm}\label{thm:spec_L_n_nu=1}
For all $n\in\N$, one has
 \begin{equation}
  \det(1-zL_{n})=\pFq{3}{2}{-n,1/2+\ii\sqrt{z-1/4}, 1/2-\ii\sqrt{z-1/4}}{1,1}{1}
 \label{eq:char_pol_L}
 \end{equation}
and
\[
 0<\mu_{1,n}<\mu_{2,n}<\dots<\mu_{n,n}<4.
\]
\end{thm}

\begin{rem}
At this point, any branch of the square root can be taken in~\eqref{eq:char_pol_L} but later on, we will always use the standard branch.
\end{rem}

\begin{proof}[Proof of Theorem~\ref{thm:spec_L_n_nu=1}]
It was observed in~\cite[Subsec.~3.1]{sta_jfa22} that, if $\nu=1$, polynomials $P_{n}$ can be expressed in terms of the continuous dual Hahn orthogonal polynomials 
 \begin{equation}
  S_{n}(x^{2};a,b,c):=(a+b)_{n}(a+c)_{n}\,\pFq{3}{2}{-n,a+\ii x,a-\ii x}{a+b,a+c}{1},
 \label{eq:cont_dual_hahn_ogp}
 \end{equation}
see~\cite[Sec.~9.3]{koe-les-swa_10}. Namely, one has the identity
 \[
  P_{n}(z;1)=\frac{1}{n!(n+1)!}\;S_{n}\left(z-\frac{1}{4};\frac{1}{2},\frac{1}{2},\frac{3}{2}\right)=
  \pFq{3}{2}{-n,\xi_{+},\xi_{-}}{1,2}{1},
 \]
where we used the temporary notation
\[
 \xi_{\pm}:=\frac{1}{2}\pm\ii\sqrt{z-\frac{1}{4}}
\]
for brevity. Next, using~\eqref{eq:char_pol_L_n_P_n} with $\nu=1$ and the very definition of the hypergeometric function, one easily verifies that
 \begin{align*}
 \det(1-zL_{n})&=(n+1)\pFq{3}{2}{-n,\xi_{+},\xi_{-}}{1,2}{1}-n\,\pFq{3}{2}{-n+1,\xi_{+},\xi_{-}}{1,2}{1}\\
 &=\sum_{k=0}^{n}\frac{(-n)_{k}(\xi_{+})_{k}(\xi_{-})_{k}}{(k!)^{3}}=\pFq{3}{2}{-n,\xi_{+},\xi_{-}}{1,1}{1}.
 \end{align*}

Next, it follows from Theorem~\ref{thm:spec_L_n} and the equality $\|L(1)\|=4$ that $0<\mu_{1,n}<\mu_{2,n}<\dots<\mu_{n,n}\leq4$. Thus, to conclude the proof, it suffices to show that $4$ is not an eigenvalue of~$L_{n}$. To this end, we use the already proven formula~\eqref{eq:char_pol_L} with $z=1/4$ and
Sheppard's transformation 
\[
 \pFq{3}{2}{-n,a,b}{c,d}{1}=\frac{(c-a)_{n}}{(c)_{n}}\pFq{3}{2}{-n,a,d-b}{a-c+1-n,d}{1},
\]
see~\cite[Eq.~7.4.4.85]{pru-bry-mar_vol3}.
It follows that 
\begin{align*}
 \det\left(1-\frac{1}{4}L_{n}\right)&=\pFq{3}{2}{-n,1/2,1/2}{1,1}{1}=\frac{\left(1/2\right)_{n}}{n!}\pFq{3}{2}{-n,1/2, 1/2}{-n+1/2,1}{1}\\
 &=\frac{\left(1/2\right)_{n}}{n!}\sum_{k=0}^{n}\frac{n(n-1)\dots(n-k+1)}{(n-1/2)(n-3/2)\dots(n-k+1/2)}\frac{(1/2)_{k}^{2}}{(k!)^{2}}.
\end{align*}
Since all terms in the last sum are positive $4$ is not a root of the characteristic polynomial of $L_{n}$ for all $n\in\N$.
\end{proof}

\begin{rem}\label{rem:cdh_polyn}
It follows from \eqref{eq:char_pol_L} and \eqref{eq:cont_dual_hahn_ogp} that the reciprocal polynomial of the characteristic polynomial of $L_{n}$ is also readily related to a particular case of the continuous dual Hahn polynomials. Namely, we have
\[
 \det(1-zL_{n})=\frac{1}{(n!)^{2}}\,S_{n}\!\left(z-\frac{1}{4};\frac{1}{2},\frac{1}{2},\frac{1}{2}\right),
\]
for all $n\in\N$ and $z\in\C$. Moreover, since the continuous dual Hahn polynomials fulfill
\begin{align*}
S_{n+1}(x;a,b,c)&=\left((n+a+b)(n+a+c)+n(n+b+c-1)-a^{2}-x\right)S_{n}(x;a,b,c) \nonumber\\
&\hskip19pt -n(n+a+b-1)(n+a+c-1)(n+b+c-1)S_{n-1}(x;a,b,c),
\end{align*}
for $n\in\N$, see~\cite[Eq.~9.3.5]{koe-les-swa_10}, the characteristic polynomial of $L_{n}$ satisfies 
the recurrence relation
\[
 (n+1)^{2}\det(1-zL_{n+1})+\left(z-2n^{2}-2n-1\right)\det(1-zL_{n})+ n^{2}\det(1-zL_{n-1})=0,
\]
for $n\geq 2$.
\end{rem}

\section{Spectral density}\label{sec:number_evls_limit}

We will prove the limit formula from Theorem~\ref{thm:l-hilbert_limit} on the number of eigenvalues of $L_{n}(\nu)$ that exceed a fixed $x\in(0,4)$, as $n\to\infty$. As a preliminary step, we need to analyze an asymptotic behavior of $\det(1-zL_{n}(\nu))$, as $n\to\infty$, in the non-oscillatory regime $z\in\C\setminus\R$. Since 
$\det(1-zL_{n}(\nu))$ is a polynomial in $z$ with real coefficients we have, for all $z\in\C$, the symmetry
\[
 \overline{\det(1-zL_{n}(\nu))}=\det(1-\overline{z}L_{n}(\nu)),
\]
where the bar stands for the complex conjugation. Therefore we can restrict the analysis to the upper half-plane $\Im z>0$. Recall the square root assumes its principle branch.

\begin{prop}\label{prop:asympt_char_pol_L}
 For $\nu\in\R\setminus(-\N_{0})$, we have
\[
\det(1-\xi L_{n}(\nu))=\left(z+\frac{1}{2}\right)\Gamma(2z)\chi(z;\nu)\,n^{z-1/2}\left[1+O\left(\frac{1}{n^{\min(1,2\Re z)}}\right)\right],
\]
as $n\to\infty$, locally uniformly in the half-plane $\Im\xi>0$, where $z=-\ii\sqrt{\xi-1/4}$ and function~$\chi$ is defined by~\eqref{eq:def_chi}.
\end{prop}

\begin{proof}
 We make use of expression~\eqref{eq:P_rel_chi_phi} to deduce an asymptotic expansion of $P_{n}$ first, and then apply~\eqref{eq:char_pol_L_n_P_n}. It follows from the right-hand side of~\eqref{eq:char_pol_L_n_P_n} that we need the first two terms from the asymptotic expansion of $P_{n}$ in order to deduce the leading term in the asymptotic formula for $\det(1-zL_{n}(\nu))$, for $n\to\infty$.
 
The computation of an asymptotic expansion of function~\eqref{eq:def_phi} is a straightforward application of the Stirling formula and the very definition of a hypergeometric series. It results in expansions
\[
 \frac{\Gamma(n+\nu+1)}{\Gamma(n+\nu+z+3/2)}=n^{-1/2-z}\left[1+O\left(\frac{1}{n}\right)\right]
\]
and
\[
 \pFq{3}{2}{1/2+z, 1/2+z, 3/2+z}{1+2z,n+\nu+z+3/2}{1}=1+O\left(\frac{1}{n}\right),
\]
for $n\to\infty$, that are local uniform in $z\in\C$. Thus
\[
 \phi_{n}(z)=\frac{n^{-1/2-z}}{\Gamma(1+2z)}\left[1+O\left(\frac{1}{n}\right)\right]
\]
and, by~\eqref{eq:P_rel_chi_phi} and the reflection formula for the Gamma function
\[
 \Gamma(z)\Gamma(1-z)=\frac{\pi}{\sin(\pi z)},
\]
we obtain
\begin{equation}
 P_{n}\left(\frac{1}{4}-z^{2};\nu\right)=\Gamma(2z)\chi(z;\nu)\,n^{z-1/2}\left[1+O\left(\frac{1}{n}\right)\right]+[z\mapsto-z],
\label{eq:P_n_asympt}
\end{equation}
for $n\to\infty$, locally uniformly in $z\in\C\setminus(\Z/2)$. Here and below, abbreviation $[z\to-z]$ stands for the term which equals the previous displayed term with $z$ replaced by $-z$. Applying~\eqref{eq:P_n_asympt} in~\eqref{eq:char_pol_L_n_P_n} yields
\[
 \det\left(1-\left(\frac{1}{4}-z^{2}\right)\!L_{n}(\nu)\right)=\left(z+\frac{1}{2}\right)\Gamma(2z)\chi(z;\nu)\,n^{z-1/2}\left[1+O\left(\frac{1}{n}\right)\right]+[z\mapsto-z],
\]
for $n\to\infty$, locally uniformly in $z\in\C\setminus(\Z/2)$. Finally, it suffices to write $z=-\ii\sqrt{\xi-1/4}$ in the last formula and realize that, if $\Im\xi>0$, then $\Re z=\Im\sqrt{\xi-1/4}>0$ for the standard branch of the square root. It means that, for $\Im \xi>0$, the leading term of the asymptotic expansion is the one of order $n^{z-1/2}$. The next term is either of order $n^{z-3/2}$ or $n^{-z-1/2}$ depending on the value of $\Re z$. Hence, after factoring out the term $n^{z-1/2}$, the remainder decays at least as $n^{-\min(1,2\Re z)}$, for $\Re z>0$.
\end{proof}

An immediate consequence of Proposition~\ref{prop:asympt_char_pol_L} is the following ratio asymptotic formula. Recall that, if a sequence of analytic functions $f_{n}$ converges locally uniformly to an analytic function $f$ on an open set $U\subset\C$, for $n\to\infty$, then $f_{n}'$ converges locally uniformly to $f'$ on $U$, for $n\to\infty$; see~\cite[Chp.~VII, Thm.~2.1]{con_78}.

\begin{cor}
For $\nu\in\R\setminus(-\N_{0})$ and $\xi\in\C\setminus\R$, we have
\begin{equation}
 \lim_{n\to\infty}\frac{1}{\log n}\frac{\partial_{\xi}\det(1-\xi L_{n}(\nu))}{\det(1-\xi L_{n}(\nu))}=\frac{\sign(\Im\xi)}{2\ii\sqrt{\xi-1/4}}.
\label{eq:ration_asympt}
\end{equation}
\end{cor}

Note the limit function in~\eqref{eq:ration_asympt} is analytic in $\C\setminus[1/4,\infty)$ and has a branch cut in $[1/4,\infty)$. This cut determines the spectral distribution of Theorem~\ref{thm:l-hilbert_limit}.

\begin{proof}[Proof of Theorem~\ref{thm:l-hilbert_limit}]
Let $\nu\in\R\setminus(-\N_{0})$ and $x\in(0,4)$ be fixed. Clearly, the number of eigenvalues of $L_{n}(\nu)$ located in $(x,4)$ coincides with the number of roots of polynomial $\xi\mapsto\det(1-\xi L_{n}(\nu))$ in $(1/4,1/x)$. By the Argument Principle, we have
 \[
  \#\{\mu\in\sigma(L_{n}(\nu)) \mid x<\mu<4\}=\frac{1}{2\pi\ii}\oint_{\gamma_{x}}\frac{\partial_{\xi}\det(1-\xi L_{n}(\nu))}{\det(1-\xi L_{n}(\nu))}\,\dd\xi,
 \]
 where $\gamma_{x}$ is a simple closed counter-clockwise oriented curve crossing the real line at the points $1/4$ and $1/x$. Then, using~\eqref{eq:ration_asympt}, one deduces 
 \[
 \lim_{n\to\infty}\frac{\#\{\mu\in\sigma(L_{n}(\nu)) \mid x<\mu<4\}}{\log n}=-\frac{1}{4\pi}\oint_{\gamma_{x}}\frac{\sign(\Im z)}{\sqrt{z-1/4}}\,\dd z.
 \]
 Lastly, by squeezing the integration curve towards the real line and taking into account the orientation, we conclude
 \[
 \lim_{n\to\infty}\frac{\#\{\mu\in\sigma(L_{n}(\nu)) \mid x<\mu<4\}}{\log n}=\frac{1}{2\pi}\int_{1/4}^{1/x}\frac{\dd y}{\sqrt{y-1/4}}=\frac{1}{\pi}\sqrt{\frac{1}{x}-\frac{1}{4}},
 \]
 which implies Theorem~\ref{thm:l-hilbert_limit}.
\end{proof}

\begin{rem}
 As kindly pointed out by one of the anonymous reviewers, one can alternatively prove Theorem~\ref{thm:l-hilbert_limit} by applying~\cite[Thm.~2.1]{pus_preprint} to the symbol $\varphi(t)=4/(1+4t^{2})$, which yields the result for $\nu=1$, and then generalize to all $\nu\in\R\setminus(-\N_{0})$ by using the fact that the Hilbert--Schmidt norm of $L_{n}(\nu)-L_{n}(1)$ is uniformly bounded. 
\end{rem}

\section{Small eigenvalues}\label{sec:small_evls}

In this section, we derive the asymptotic formula of Theorem~\ref{thm:l-hilbert_small_evls} for small eigenvalues of $L_{n}(\nu)$, for $\nu>0$, as $n\to\infty$. It turns out that the desired result can be deduced with the aid of a general asymptotic formula for large roots of polynomials generated by a special second order recurrence. We recall the respective theorem which is due to A.~Maté, P.~Nevai, and V. Totik~\cite{mat-nev-tot_jlms86}; see also~\cite[Thm.~6.1]{che-ism_jcam97}.

\begin{thm}[Maté, Nevai, Totik]\label{thm:mate-nevai-totik}
 Let $\{\gamma_{n}\}_{n=1}^{\infty}$ be a positive sequence such that
 \[
  \gamma_{n}=c^{2}n^{2\delta}\left(1+o\left(n^{-2/3}\right)\right)\!, \quad n\to\infty,
 \]
 where $c,\delta>0$. Then zeros $x_{1,n}<x_{2,n}<\dots<x_{n,n}$ of polynomial $Q_{n}$ defined recursively by
 \begin{equation}
   Q_{0}(x)=1, \quad Q_{1}(x)=x, \quad\mbox{ and }\quad Q_{n+1}(x)=xQ_{n}(x)-\gamma_{n}Q_{n-1}(x), \quad n\in\N,
 \label{eq:Q_n_recur}
 \end{equation}
 fulfill 
 \[
  x_{n-j+1,n}=2cn^{\delta}\left[1-6^{-1/3}\delta^{2/3}i_{j}\,n^{-2/3}+o\left(n^{-2/3}\right)\right], \quad n\to\infty,
 \]
 for $j\in\N$ fixed, where $i_{1}<i_{2}<\dots$ are positive zeros of the Airy function
 \begin{equation}
 A(x):=\frac{\pi}{3}\sum_{m=0}^{\infty}\left(\frac{1}{\Gamma(m+2/3)}+\frac{x}{3\Gamma(m+4/3)}\right)\frac{(-x/3)^{3m}}{m!},
 \label{eq:airy}
 \end{equation}
 (see~\cite[\S~1.81 and \S~6.32]{sze_39}).
\end{thm}

\begin{rem}
To be precise, the original theorem of A.~Maté, P.~Nevai, and V. Totik from~\cite{mat-nev-tot_jlms86} is formulated in terms of the orthonormal polynomials $p_{n}$ rather than the monic orthogonal polynomials $Q_{n}$ which are used in Theorem~\ref{thm:mate-nevai-totik}. The two families are related by the formula
\[
 Q_{n}(x)=\left(\prod_{j=1}^{n}\gamma_{j}^{1/2}\right)p_{n}(x), \quad n\in\N_{0},
\]
and hence have the same roots. Moreover, if we denote $\alpha_{n}:=\sqrt{\gamma_{n}}$, it follows from~\eqref{eq:Q_n_recur} that polynomials $p_{n}$ fulfill the recurrence
\[
 xp_{n}(x)=\alpha_{n+1}p_{n+1}(x)+\alpha_{n}p_{n-1}(x), \quad n\in\N,
\]
which appears in~\cite{mat-nev-tot_jlms86}.
\end{rem}

Given a sequence $\{b_{n}\}_{n=0}^{\infty}$, consider the family of polynomials defined recursively by equations
\[
 Q_{0}(x)=1, \quad Q_{1}(x)=x, \quad\mbox{ and }\quad Q_{n+1}(x)=xQ_{n}(x)-b_{\left\lfloor(n-1)/2\right\rfloor}Q_{n-1}(x), \quad n\in\N.
\]
Here and below, $\lfloor x\rfloor$ is the greatest integer less than or equal to $x\in\R$. Then one readily checks that
\begin{equation}
 Q_{2n+2}(x)=(x^{2}-b_{n-1}-b_{n})Q_{2n}(x)-b_{n-1}^{2}Q_{2n-2}(x),
\label{eq:Q_2n_recur}
\end{equation}
which implies the identity
\begin{equation}
 Q_{2n}(x)=\det(x^{2}-B_{n}),
\label{eq:Q_n_rel_B_n}
\end{equation}
for all $n\in\N$, where Jacobi matrix $B_{n}$ is defined by~\eqref{eq:matrix_B}.

\begin{proof}[Proof of Theorem~\ref{thm:l-hilbert_small_evls}]
 Let us denote by $B_{n}(\nu)$ the Jacobi matrix given by~\eqref{eq:matrix_B} with $b_{n}=(n+\nu)(n+1+\nu)$ and the eigenvalues of $B_{n}(\nu)$ by $\beta_{1,n}(\nu)<\beta_{2,n}(\nu)<\dots<\beta_{n,n}(\nu)$. First, we relate eigenvalues of $B_{n}(\nu)$ and $L_{n}(\nu)$. It follows from~\eqref{eq:A_rel_B_rank-one} that $B_{n}(\nu)$ is a rank-one perturbation of $L_{n}^{-1}(\nu)$, namely
 \[
  B_{n}(\nu)=L_{n}^{-1}(\nu)+(n+\nu-1)^{2}e_{n}e_{n}^{T}.
 \]
 Therefore $B_{n}(\nu)\geq L_{n}^{-1}(\nu)$, which implies
\begin{equation}
 \beta_{j,n}(\nu)\geq\frac{1}{\mu_{n-j+1,n}(\nu)},
 \label{eq:ineq_from_weyl}
\end{equation}
for all $j\in\{1,2,\dots,n\}$, by Weyl's inequality. Further, since $B_{n-1}(\nu)$ is the $(n-1)\times (n-1)$ section of $L_{n}^{-1}(\nu)$, we also know that
\begin{equation}
 \frac{1}{\mu_{n,n}(\nu)}\leq\beta_{1,n-1}(\nu)\leq \frac{1}{\mu_{n-1,n}(\nu)}\leq\dots\leq \frac{1}{\mu_{2,n}(\nu)}\leq\beta_{n-1,n-1}(\nu)\leq\frac{1}{\mu_{1,n}(\nu)} 
\label{eq:ineq_from_cauchy}
\end{equation}
by the Cauchy interlacing theorem for eigenvalues of Hermitian matrices. Combining~\eqref{eq:ineq_from_weyl} and~\eqref{eq:ineq_from_cauchy}, we get
\begin{equation}
\frac{1}{\beta_{n-j+1,n}(\nu)}\leq\mu_{j,n}(\nu)\leq\frac{1}{\beta_{n-j,n-1}(\nu)},
\label{eq:mu_beta_bounds}
\end{equation}
for $j\in\{1,2,\dots,n-1\}$.

In view of~\eqref{eq:Q_n_rel_B_n} and~\eqref{eq:Q_2n_recur}, one checks the polynomials
\[
 Q_{2n}(x):=\det(x^{2}-B_{n}(\nu))
\]
coincide with the even index polynomials generated by recurrence~\eqref{eq:Q_n_recur} with
\[
\gamma_{n}:=b_{\left\lfloor(n-1)/2\right\rfloor}=\left(\left\lfloor\frac{n-1}{2}\right\rfloor+\nu\right)\left(\left\lfloor\frac{n+1}{2}\right\rfloor+\nu\right).
\]
Since
\[
 \gamma_{n}=\frac{n^{2}}{4}\left(1+O\left(\frac{1}{n}\right)\right)\!,\quad n\to\infty,
\]
i.e., $c=1/2$ and $\delta=1$, Theorem~\ref{thm:mate-nevai-totik} tells us that 
\[
 \beta_{n-j+1,n}(\nu)=x_{2n-j+1,2n}^{2}=4n^{2}\left[1-\frac{i_{j}}{\sqrt[3]{3}}\,n^{-2/3}+o\left(n^{-2/3}\right)\right], \quad n\to\infty,
\]
provided that $\nu>0$, which assures $\gamma_{n}>0$ for $n\in\N$.
Notice that the same asymptotic formula holds true for $\beta_{n-j,n-1}(\nu)$. Consequently, with the aid of~\eqref{eq:mu_beta_bounds}, we arrive at the asymptotic formula from Theorem~\ref{thm:l-hilbert_small_evls}.
\end{proof}

\section{Large eigenvalues}\label{sec:large_evls}

Throughout this section, we consider only the case $\nu=1$. Our final goal is an asymptotic analysis of large eigenvalues of $L_{n}$ for $n\to\infty$.  To do so, we first investigate an asymptotic behavior of the characteristic polynomial of $L_{n}$ in the oscillatory region, more precisely, in a left neighborhood of $4$; see Theorem~\ref{thm:spec_L_n_nu=1}. We substitute for 
\[
z=\frac{1}{4}+\xi^{2}
\]
in~\eqref{eq:char_pol_L} and study an asymptotic behavior of the auxiliary polynomial 
\begin{equation}
 q_{n}(\xi):=\xi\det\!\left(1-\left(\frac{1}{4}+\xi^{2}\right)\!L_{n}\right)=\xi\,\pFq{3}{2}{-n,1/2+\ii\xi,1/2-\ii\xi}{1,1}{1}.
\label{eq:def_q_n}
\end{equation}
Since we aim at an asymptotic analysis of the roots of $q_{n}$ a special care has to be taken for the remainder in the asymptotic expansion of~$q_{n}$. We prove the following formula.

\begin{prop}\label{prop:q_n_asympt}
 Let $K>0$, $0<L<1/2$, and $\mathcal{S}_{K,L}:=\{\xi\in\C\mid |\Re\xi|\leq K \;\wedge\;|\Im\xi|\leq L\}$.  Then there exists a constant $C_{K,L}>0$ such that, for all $\xi\in\mathcal{S}_{K,L}$ and $n\in\N$, one has
 \begin{align*}
 q_{n}(\xi)=\frac{1}{\sqrt{n}}\left[\frac{F(\xi)-F(-\xi)}{2\ii}\,\cos(\xi\log n)+\frac{F(\xi)+F(-\xi)}{2}\,\sin(\xi\log n)\right]+R_{n}(\xi),
 \end{align*}
 where 
 \[
  \left|R_{n}(\xi)\right|\leq\frac{C_{K,L}}{n}
 \]
 and
 \begin{equation}
 F(\xi):=\frac{\Gamma(1+2\ii\xi)}{\Gamma^{3}(1/2+\ii\xi)}.
 \label{eq:def_func_F}
 \end{equation}
\end{prop}

\begin{rem}
 Since 
 \[
  q_{n}(\xi)=\frac{\xi}{(n!)^{2}}\,S_{n}\!\left(\xi^{2};\frac{1}{2},\frac{1}{2},\frac{1}{2}\right),
 \]
 see Remark~\ref{rem:cdh_polyn}, one may hope to find asymptotic formulas for small zeros of the continuous dual Hahn polynomials in the extensive literature on orthogonal polynomials. Although asymptotic formulas for extreme zeros are available for various families of orthogonal polynomials, see for instance~\cite{ism_05} and references therein, this seems to be not the case for the continuous dual Hahn polynomials to author's best knowledge.
\end{rem}

The proof of Proposition~\ref{prop:q_n_asympt} essentially uses ideas of the Darboux method, see \cite[Sec.~8.9]{olv_97}. However, we need to keep track on the dependence on $\xi$ in the error term 
which requires a~more detailed analysis of the remainder.

A starting point of the method is the generating function formula
\begin{equation}
f_{\xi}(t):=\xi\,(1-t)^{-1/2+\ii\xi}\,\pFq{2}{1}{1/2+\ii\xi, 1/2+\ii\xi}{1}{t}=\sum_{n=0}^{\infty}q_{n}(\xi)t^{n},
\label{eq:gener_func_q_n}
\end{equation}
for $|t|<1$, which is a particular case of~\cite[Eq.~9.3.12]{koe-les-swa_10}. For $|\Im\xi|<1/2$, the Gauss hypergeometric function in~\eqref{eq:gener_func_q_n} is an analytic function in the open unit disk and continuous to the boundary unit circle except possibly the point $t=1$, where it may have a singularity, see~\cite[Chap.~15]{dlmf}. If $\Im\xi>-1/2$, the point $t=1$ is a singular point of function~\eqref{eq:gener_func_q_n} that is located most closely to the origin and determines the asymptotic behavior of $q_{n}$, for $n$ large. A behavior of function~\eqref{eq:gener_func_q_n} in a~neighborhood of the singular point $t=1$ reveals the identity
\begin{equation}
f_{\xi}(t)=(1-t)^{-1/2+\ii\xi}\,\frac{\ii\Gamma(1-2\ii\xi)}{2\Gamma^{2}(1/2-\ii\xi)}\,\pFq{2}{1}{1/2+\ii\xi, 1/2+\ii\xi}{1+2\ii\xi}{1-t}-[\xi\mapsto-\xi],
\label{eq:f_transf1}
\end{equation}
which holds if $|\arg t|<\pi$, $|\arg(1-t)|<\pi$, and is a consequence of transformation~\cite[Eq.~15.10.21]{dlmf} and the identity $\Gamma(1\pm2\ii\xi)=\pm2\ii\xi\Gamma(\pm2\ii\xi)$.
The leading term in the asymptotic expansion of~\eqref{eq:f_transf1}, as $t\to1$ from the unit disk, provides us with the comparison function
\begin{equation}
g_{\xi}(t):=(1-t)^{-1/2+\ii\xi}\,\frac{\ii\Gamma(1-2\ii\xi)}{2\Gamma^{2}(1/2-\ii\xi)}-(1-t)^{-1/2-\ii\xi}\,\frac{\ii\Gamma(1+2\ii\xi)}{2\Gamma^{2}(1/2+\ii\xi)}
\label{eq:def_g_xi}
\end{equation}
to the Darboux method.

Difference $f_{\xi}-g_{\xi}$ is an analytic function in the unit disk and continuous to the boundary, if  $|\Im\xi|<1/2$. Hence, by Cauchy's integral formula, we have
\begin{equation}
 q_{n}(\xi)-c_{n}(\xi)=\frac{1}{2\pi}\int_{-\pi}^{\pi}e^{-\ii n\theta}\left(f_{\xi}(e^{\ii\theta})-g_{\xi}(e^{\ii\theta})\right)\dd\theta,
 \label{eq:diff_cauchy_int_formula}
\end{equation}
where $c_{n}(\xi)$ are coefficients from the Maclaurin series
\[
 g_{\xi}(t)=\sum_{n=0}^{\infty}c_{n}(\xi)t^{n},
\]
for $|t|<1$. Since $f_{\xi}'-g_{\xi}'=O\!\left((1-t)^{-1/2-\Im\xi}\right)$, as $t\to1$, it is still integrable on the unit circle provided that $|\Im\xi|<1/2$. Therefore we can integrate by parts in~\eqref{eq:diff_cauchy_int_formula}, which yields
\begin{equation}
 q_{n}(\xi)-c_{n}(\xi)=\frac{1}{2\pi n}\int_{-\pi}^{\pi}e^{-\ii(n-1)\theta}\left(f_{\xi}'(e^{\ii\theta})-g_{\xi}'(e^{\ii\theta})\right)\dd\theta.
 \label{eq:diff_q_n_c_n_integral}
\end{equation}
It turns out that the integral in \eqref{eq:diff_q_n_c_n_integral} is uniformly bounded in $\mathcal{S}_{K,L}$ with $K>0$ and $0<L<1/2$ fixed. It is the statement of the following lemma whose proof is straightforward but technical and therefore it is postponed to Appendix~\ref{subsec:a.1}.

\begin{lem}\label{lem:remainder_estim_aux}
 For $K>0$ and $0<L<1/2$, there is $C_{K,L}'>0$ such that 
 \[
  \left|\int_{-\pi}^{\pi}e^{-\ii(n-1)\theta}\left(f_{\xi}'(e^{\ii\theta})-g_{\xi}'(e^{\ii\theta})\right)\dd\theta\right|\leq C_{K,L}',
 \]
 for all $\xi\in\mathcal{S}_{K,L}$ and $n\in\N$.
\end{lem}

Now, we are in position to verify the formula from Proposition~\ref{prop:q_n_asympt}.

\begin{proof}[Proof of Proposition~\ref{prop:q_n_asympt}]
Fix $K>0$ and $0<L<1/2$. It follows from~\eqref{eq:diff_q_n_c_n_integral} and Lemma~\ref{lem:remainder_estim_aux} that 
\[
 |q_{n}(\xi)-c_{n}(\xi)|\leq\frac{C_{K,L}'}{2\pi n},
\]
for all $\xi\in\mathcal{S}_{K,L}$ and $n\in\N$. Consequently, it suffices to prove the claim of Proposition~\ref{prop:q_n_asympt} with $q_{n}$ replaced by $c_{n}$.

For coefficients from the Maclaurin series of~\eqref{eq:def_g_xi}, one readily gets
\begin{equation}
 c_{n}(\xi)=\frac{\ii\Gamma(1-2\ii\xi)}{2\Gamma^{3}(1/2-\ii\xi)}\frac{\Gamma(n+1/2-\ii\xi)}{\Gamma(n+1)}-[\xi\mapsto-\xi],
\label{eq:c_n_coeff}
\end{equation}
for all $n\in\N_{0}$ and $\xi\in\mathcal{S}_{K,L}$. Next, we can use known error bounds in asymptotic expansions for the ratio of two Gamma functions~\cite{fre_siamjma92}. Namely, it follows from~\cite[Eqs.~1.4 and~1.5]{fre_siamjma92} that there is a constant $C_{K,L,1}>0$ such that, for all $n\in\N$ and $\xi\in\mathcal{S}_{K,L}$, it holds
\begin{equation}
\frac{\Gamma(n+1/2\pm\ii\xi)}{\Gamma(n+1)}=n^{-1/2\pm\ii\xi}+\rho_{n}(\pm\xi),
\label{eq:frenzen_bound}
\end{equation}
where $|\rho_{n}(\pm\xi)|\leq C_{K,L,1}/n$. Further, there exists a constant $C_{K,L,2}>0$ such that, for all $\xi\in\mathcal{S}_{K,L}$, one has

\begin{equation}
\left|\frac{\ii\Gamma(1\pm 2\ii\xi)}{2\Gamma^{3}(1/2\pm\ii\xi)}\right|\leq C_{K,L,2},
\label{eq:func_bound_aux}
\end{equation}
since the function on the left-hand side is continuous in the compact set~$\mathcal{S}_{K,L}$. By combining~\eqref{eq:c_n_coeff}, \eqref{eq:frenzen_bound}, and~\eqref{eq:func_bound_aux}, we obtain formula
\[
 c_{n}(\xi)=\frac{1}{\sqrt{n}}\left[\frac{\ii\Gamma(1-2\ii\xi)}{2\Gamma^{3}(1/2-\ii\xi)}\,n^{-\ii\xi}-\frac{\ii\Gamma(1+2\ii\xi)}{2\Gamma^{3}(1/2+\ii\xi)}\,n^{\ii\xi}\right]+R_{n}(\xi),
\]
for all $n\in\N$ and $\xi\in\mathcal{S}_{K,L}$, where $|R_{n}(\xi)|\leq C_{K,L}/n$ for $C_{K,L}:=2C_{K,L,1}C_{K,L,2}$. 
Finally, with the aid of the identity
\[
\quad n^{\pm\ii\xi}=\cos(\xi\log n)\pm\ii\sin(\xi \log n),
\]
we readily arrive at the statement of Proposition~\ref{prop:q_n_asympt} with $q_{n}$ replaced by $c_{n}$.
\end{proof}


Having Proposition~\ref{prop:q_n_asympt}, we can readily derive asymptotic expansions for zeros of $q_{n}$, for $n\to\infty$. Notice that $q_{n}$ is an odd polynomial of degree $2n+1$ with simple zeros distributed symmetrically with respect to the origin, see~\eqref{eq:def_q_n} and Theorem~\ref{thm:spec_L_n_nu=1}. The same holds true for the scaled polynomial 
\[
 Q_{n}(\xi):=\sqrt{n}\,q_{n}\!\left(\frac{\xi}{\log n}\right),
\]
whose positive roots we denote by
 \[
  0<\xi_{1,n}<\xi_{2,n}<\dots<\xi_{n,n}.
 \]
Then, by~\eqref{eq:def_q_n}, we have
\begin{equation}
 \mu_{n-j+1,n}=\frac{4}{1+\frac{4\xi_{j,n}^{2}}{\log^{2}n}},
\label{eq:mu_rel_xi}
\end{equation}
for all $j\in\{1,\dots,n\}$.

\begin{proof}[Proof of Theorem~\ref{thm:l-hilbert_large_evls}] It follows from Proposition~\ref{prop:q_n_asympt} that 
\[
 \lim_{n\to\infty}Q_{n}(\xi)=\Im\!\left(F(0)\right)\cos\xi + \Re\!\left(F(0)\right)\sin\xi=\frac{\Gamma(1)}{\Gamma^{3}(1/2)}\sin\xi=\frac{1}{\pi^{3/2}}\sin\xi
\]
locally uniformly in $\C$. Since the zeros of the sine function are simple and located at points $\pi j$, for $j\in\Z$, Hurwitz's theorem~\cite[Chp.~VII, Thm.~2.5]{con_78} implies that, for any $j\in\Z$ and $\epsilon\in(0,\pi)$ fixed, polynomial $Q_{n}$ has exactly one simple root in $(\pi j-\epsilon,\pi j+\epsilon)$ for all $n$ sufficiently large. It follows that, for $j\in\N$ fixed, we have
\[
 \xi_{j,n}=\pi j+o(1), \quad n\to\infty.
\]

Similarly, we can deduce the next term in the asymptotic expansion of $\xi_{j,n}$, for $n\to\infty$. We write 
 \begin{equation}
  \xi_{j,n}=\pi j + \eta_{j,n},
 \label{eq:xi_expand_first_order}
 \end{equation}
 where $\eta_{j,n}\to0$, as $n\to\infty$. Then, by applying Proposition~\ref{prop:q_n_asympt} in the equation
 \[
 \log(n)\,Q_{n}(\xi_{j,n})=0,
 \]
 we obtain the condition
 \begin{equation}
 \lim_{n\to\infty} \log(n)\Im F\left(\frac{\pi j + \eta_{j,n}}{\log n}\right)\cos(\eta_{j,n})+\log (n)\Re F\left(\frac{\pi j + \eta_{j,n}}{\log n}\right)\sin(\eta_{j,n})=0.
 \label{eq:fund_limit_eq_second}
 \end{equation}
 since the remainder term even when multiplied by $\log n$ (or any power of $\log n$) vanishes, as $n\to\infty$, because it decays at least as fast as $1/\sqrt{n}$ uniformly in $\xi$. In view of the limit formulas
 \[
 \lim_{n\to\infty} \log(n)\Im F\left(\frac{\pi j + \eta_{j,n}}{\log n}\right)=\pi j \Im F'(0), \quad 
 \lim_{n\to\infty} \Re F\left(\frac{\pi j + \eta_{j,n}}{\log n}\right)=\Re F(0),
 \]
 and $\lim_{n\to\infty}\cos(\eta_{j,n})=1$, one deduces from~\eqref{eq:fund_limit_eq_second} that 
 \[
 \lim_{n\to\infty}\log(n)\sin(\eta_{j,n})=\lim_{n\to\infty}\log(n)\eta_{j,n}=-\frac{\pi j\Im F'(0)}{\Re F(0)}.
 \]
 The first two coefficients from the Maclaurin expansion of function~\eqref{eq:def_func_F} are
 \begin{equation}
  F(0)=\frac{1}{\pi^{3/2}} \quad\mbox{ and }\quad F'(0)=\frac{\ii\left(\gamma+6\log 2\right)}{\pi^{3/2}},
 \label{eq:F_F'_id}
 \end{equation}
 as it follows from~\eqref{eq:def_func_F} and the known special values of the Gamma function and its derivative
 \[
  \Gamma(1)=1, \quad \Gamma(1/2)=\sqrt{\pi}, \quad \Gamma'(1)=-\gamma, \,\,\mbox{ and }\,\,
  \Gamma'(1/2)=-\sqrt{\pi}(\gamma+2\log 2),
 \]
 see~\cite[\S~5.4]{dlmf}. Recalling also~\eqref{eq:xi_expand_first_order}, we arrive at the asymptotic formula
 \begin{equation}
  \xi_{j,n}=\pi j-\frac{\pi j\left(\gamma+6\log 2\right)}{\log n}+o\left(\frac{1}{\log n}\right), \quad n\to\infty,
 \label{eq:xi_asympt_expand}
 \end{equation}
 which holds true for any fixed $j\in\N$. It is also clear that the remainder $o(1/\log n)$ can be replaced by $O(1/\log^{2} n)$, which one would verify by repeating the same computational procedure (and determine the next term in the expansion).

Finally, a direct computation using~\eqref{eq:mu_rel_xi} and~\eqref{eq:xi_asympt_expand} yields
 \[
 \mu_{n-j+1,n}= 4-\frac{16\pi^{2}j^{2}}{\log^{2}n}+\frac{32\pi^{2}j^{2}\left(\gamma+6\log 2\right)}{\log^{3}n}+O\left(\frac{1}{\log^{4}n}\right), \quad n\to\infty,
 \]
 for $j\in\N$, which is the claim of Theorem~\ref{thm:l-hilbert_large_evls}. 
\end{proof}

It is clear from the proof of Theorem~\ref{thm:l-hilbert_large_evls} that the zeros $\xi_{j,n}$ and hence also eigenvalues $\mu_{n-j+1,n}$ can be asymptotically expanded to an arbitrary negative power of $\log n$. The computational procedure is purely algebraic, however, higher order terms become quickly complicated. In addition, we have not observed any particular structure in the coefficients. For the sake of curiosity, we compute two more terms in the asymptotic formula for the best constant in the truncated Hardy inequality $\mu_{n,n}=\|L_{n}\|$, as $n\to\infty$. The formula reads
\begin{align}
 \|L_{n}\|=4-\frac{16\pi^{2}}{\log^{2}n}+\frac{32\pi^{2}\kappa}{\log^{3}n}&-\frac{16\pi^{2}(3\kappa^{2}-4\pi^{2})}{\log^{4}n}\nonumber\\
 &+\frac{32\pi^{2}\left[6\kappa(\kappa^{2}-4\pi^{2})-13\pi^{2}\zeta(3)\right]}{3\log^{5}n}+O\left(\frac{1}{\log^{6}n}\right),
\label{eq:norm_asympt}
\end{align}
for $n\to\infty$, where $\kappa:=\gamma+6\log 2$. Details of the computational procedure are briefly described in Appendix~\ref{subsec:a.2}.

\section*{Acknowledgement}
The author wishes to acknowledge gratefully partial support from grant No.~20-17749X of the Czech Science Foundation.

\setcounter{section}{1}
\renewcommand{\thesection}{\Alph{section}}
\setcounter{equation}{0} \renewcommand{\theequation}{\Alph{section}.\arabic{equation}}

\section*{Appendix}

\subsection{Proof of Lemma~\ref{lem:remainder_estim_aux}}\label{subsec:a.1}

We fix $K>0$, $0<L<1/2$, and recall the notation 
\[
\mathcal{S}_{K,L}=\{\xi\in\C \mid |\Re\xi|\leq K \; \wedge \; |\Im\xi|\leq L\}. 
\]
First, we write the integral 
\[
  \int_{-\pi}^{\pi}e^{-\ii(n-1)\theta}\left(f_{\xi}'(e^{\ii\theta})-g_{\xi}'(e^{\ii\theta})\right)\dd\theta
\]
as a sum of two integrals for $|\theta|<\varepsilon$ and $\varepsilon<|\theta|<\pi$, and estimate the two integrals separately. Constant $\varepsilon$ can be chosen any number such that $0<\varepsilon<\pi/3$ which will be needed to guarantee that $|1-e^{\ii\varepsilon}|<1$. For the sake of concreteness, we set $\varepsilon:=2\arcsin(1/4)\approx0.505$, then $|1-e^{\ii\varepsilon}|=1/2$. We will also make use of the simple bounds
\begin{equation}
 |z|^{-\Im\xi} e^{-\pi|\Re \xi|}\,\leq\left|z^{\ii\xi}\right|\leq |z|^{-\Im\xi}\,e^{\pi|\Re \xi|},
 \label{eq:complex_power_aux}
\end{equation}
which holds for any $z\in\C\setminus\{0\}$ and $\xi\in\C$.

1) The integration for $\varepsilon<|\theta|<\pi$: By differentiation of~\eqref{eq:gener_func_q_n} with respect to~$t$ and using identity~\cite[Eq.~15.5.1]{dlmf}, one gets
\begin{align}
 f_{\xi}'(t)&=\xi\left(\frac{1}{2}-\ii\xi\right)(1-t)^{-3/2+\ii\xi}\,\pFq{2}{1}{1/2+\ii\xi, 1/2+\ii\xi}{1}{t}\nonumber\\
 			&+\xi\left(\frac{1}{2}+\ii\xi\right)^{2}(1-t)^{-1/2+\ii\xi}\,\pFq{2}{1}{3/2+\ii\xi, 3/2+\ii\xi}{2}{t}.
			\label{eq:f_der_in_proof}
\end{align}
Both ${}_{2}F_{1}$-functions are bounded. Indeed, with the aid of the integral representation~\cite[Eq.~15.6.1]{dlmf}
\[
 \pFq{2}{1}{a,b}{c}{t}=\frac{\Gamma(c)}{\Gamma(b)\Gamma(c-b)}\int_{0}^{1}\frac{u^{b-1}(1-u)^{c-b-1}}{(1-tu)^{a}}\dd u,
\]
which holds true for $\mathrm{Arg}\,(1-t)<\pi$ and $\Re c > \Re b >0$, one can estimate
\[
\left|\pFq{2}{1}{1/2+\ii\xi, 1/2+\ii\xi}{1}{t}\right|\leq\frac{1}{|\Gamma(1/2+\ii\xi)\Gamma(1/2-\ii\xi)|}\int_{0}^{1}\frac{u^{-1/2-\Im\xi}(1-u)^{-1/2+\Im\xi}}{\left|(1-tu)^{1/2+\ii\xi}\right|}\dd u,
\]
provided that $|\Im\xi|<1/2$.
Since
\[
 \min_{|\theta|\in[\varepsilon,\pi], u\in[0,1]}\left|1-e^{\ii\theta}u\right|=\sin\varepsilon=\frac{\sqrt{15}}{8}>\frac{1}{4},
\]
and in view of~\eqref{eq:complex_power_aux}, we have
\[
 \left|(1-tu)^{1/2+\ii\xi}\right|\geq |1-tu|^{1/2-\Im\xi}\, e^{-\pi|\Re \xi|}\geq 4^{-1/2-L} e^{-\pi K}\geq\frac{1}{4}e^{-\pi K},
\]
for all $u\in[0,1]$, $t=e^{\ii\theta}$ with $|\theta|\in[\varepsilon,\pi]$, and $\xi\in\mathcal{S}_{K,L}$. Hence
\[
\left|\pFq{2}{1}{1/2+\ii\xi, 1/2+\ii\xi}{1}{t}\right|\leq\frac{4e^{\pi K}}{|\Gamma(1/2+\ii\xi)\Gamma(1/2-\ii\xi)|}\int_{0}^{1}u^{1/2-\Im\xi}(1-u)^{1/2+\Im\xi}\dd u,
\]
for $\xi\in\mathcal{S}_{K,L}$. Using the integral identity
\[
 \int_{0}^{1}u^{-1/2-\alpha}(1-u)^{-1/2+\alpha}\dd u = \frac{\pi}{\cos(\pi\alpha)},
\]
which holds true for $|\Re\alpha|<1/2$, and using that the resulting function in the upper bound is continuous in $\mathcal{S}_{K,L}$, we verify that, for all $\xi\in\mathcal{S}_{K,L}$,
\[
\left|\pFq{2}{1}{1/2+\ii\xi, 1/2+\ii\xi}{1}{t}\right|\leq C_{1}',
\]
for a constant $C_{1}'>0$. Here and below we suppress the explicit dependence of constants $C_{i}'=C_{K,L,i}'$, $i=1,2,3,\dots$, on $K$ and $L$ in the notation for brevity.
Very analogously, one also checks that 
\[
\left|\pFq{2}{1}{3/2+\ii\xi, 3/2+\ii\xi}{2}{t}\right|\leq C_{2}',
\]
for all $\xi\in\mathcal{S}_{K,L}$ and a constant $C_{2}'>0$. Plugging these estimates into~\eqref{eq:f_der_in_proof} we readily derive the upper bound
\[
\left|f_{\xi}'(t)\right|\leq C_{1}'\,2^{3/2+\Im\xi}\,e^{\pi|\Re\xi|}\,|\xi|\left|\frac{1}{2}-\ii\xi\right|+
C_{2}'\,2^{1/2+\Im\xi}\,e^{\pi|\Re\xi|}\,|\xi|\left|\frac{1}{2}+\ii\xi\right|^{2},
\]
for $t=e^{\ii\theta}$, where $|\theta|\in[\varepsilon,\pi]$, and $\xi\in\mathcal{S}_{K,L}$. The right-hand side is obviously bounded on $\mathcal{S}_{K,L}$ by a constant, say $C_{3}'>0$. Thus, we get 
\begin{equation}
 \int_{\varepsilon<|\theta|<\pi}\left|f_{\xi}'(e^{\ii\theta})\right|\dd\theta\leq2\pi C_{3}'.
\label{eq:int_eps_pi_f_bound}
\end{equation}

Next, with the aid of~\eqref{eq:complex_power_aux}, we can estimate the expression for the derivative of~\eqref{eq:def_g_xi} as follows:
\[
 \left|g_{\xi}'(t)\right|\leq\left|\frac{1}{2}-\ii\xi\right|\,2^{3/2+\Im\xi}\,e^{\pi|\Re\xi|}\left|\frac{\ii\Gamma(1-2\ii\xi)}{2\Gamma^{2}(1/2-\ii\xi)}\right|+[\xi\mapsto-\xi],
\]
for all $t=e^{\ii\theta}$, where $|\theta|\in[\varepsilon,\pi]$, and $\xi\in\mathcal{S}_{K,L}$. As the right-hand side is again a continuous function on $\mathcal{S}_{K,L}$, we verify that there is a constant $C_{4}'>0$ such that 
\begin{equation}
 \int_{\varepsilon<|\theta|<\pi}\left|g_{\xi}'(e^{\ii\theta})\right|\dd\theta\leq2\pi C_{4}',
\label{eq:int_eps_pi_g_bound}
\end{equation}
for all $\xi\in\mathcal{S}_{K,L}$.

In total, estimates~\eqref{eq:int_eps_pi_f_bound} and~\eqref{eq:int_eps_pi_g_bound} imply
\begin{equation}
\left|\int_{\varepsilon<|\theta|<\pi}e^{-\ii(n-1)\theta}\left(f_{\xi}'(e^{\ii\theta})-g_{\xi}'(e^{\ii\theta})\right)\dd\theta\right|\leq \int_{\varepsilon<|\theta|<\pi}\left(\left|f_{\xi}'(e^{\ii\theta})\right|+\left|g_{\xi}'(e^{\ii\theta})\right|\right)\dd\theta\leq C_{5}',
\label{eq:estim1_inproof}
\end{equation}
for all $n\in\N$ and $\xi\in\mathcal{S}_{K,L}$, where $C_{5}':=2\pi(C_{3}'+C_{4}')$.

2) The integration for $|\theta|<\varepsilon$: In view of~\eqref{eq:f_transf1} and~\eqref{eq:def_g_xi}, we have
\[
 f_{\xi}(t)-g_{\xi}(t)=\frac{\ii\Gamma(1-2\ii\xi)}{2\Gamma^{2}(1/2-\ii\xi)}\sum_{k=1}^{\infty}\frac{(1/2+\ii\xi)_{k}^{2}}{k!\,(1+2\ii\xi)_{k}}(1-t)^{k-1/2+\ii\xi}-[\xi\mapsto-\xi].
\]
When differentiated with respect to $t$, we obtain
\begin{align*}
 f_{\xi}'(t)-g_{\xi}'(t)=&-\frac{\ii\Gamma(1-2\ii\xi)}{2\Gamma^{2}(1/2-\ii\xi)}\bigg(\frac{(1/2+\ii\xi)^{2}}{2(1-t)^{1/2-\ii\xi}}\\
 &\hskip68pt+\sum_{k=2}^{\infty}\frac{(1/2+\ii\xi)_{k}^{2}(k-1/2+\ii\xi)}{k!\,(1+2\ii\xi)_{k}}(1-t)^{k-3/2+\ii\xi}\bigg)
 -[\xi\mapsto-\xi].
\end{align*}
Since $|1-t|\leq1/2$, for all $t=e^{\ii\theta}$ with $|\theta|\in[0,\varepsilon]$, and using also~\eqref{eq:complex_power_aux} and the upper bound $|\Im\xi|<1/2$ for $\xi\in\mathcal{S}_{K,L}$, the sum can be majorized by a convergent series
\begin{align*}
\left|\sum_{k=2}^{\infty}\frac{(1/2+\ii\xi)_{k}^{2}(k-1/2+\ii\xi)}{k!\,(1+2\ii\xi)_{k}}(1-t)^{k-3/2+\ii\xi}\right|\leq e^{\pi K}\sum_{k=1}^{\infty}\frac{\sqrt{(k+1)^{2}+K^{2}}}{k!(k+1)!}2^{-k}\prod_{j=1}^{k+1}\left[j^{2}+K^{2}\right].
\end{align*}
We denote the right-hand side by $C_{6}'$ and deduce the upper bound
\begin{align*}
 &\left|\int_{-\varepsilon}^{\varepsilon}e^{-\ii(n-1)\theta}\left(f_{\xi}'(e^{\ii\theta})-g_{\xi}'(e^{\ii\theta})\right)\dd\theta\right|\leq\int_{-\varepsilon}^{\varepsilon}\left|f_{\xi}'(e^{\ii\theta})-g_{\xi}'(e^{\ii\theta})\right|\dd\theta\\
 &\hskip28pt\leq\frac{|\Gamma(1-2\ii\xi)|}{4|\Gamma(1/2-\ii\xi)|^{2}}\left(|1/2+\ii\xi|^{2}\int_{-\varepsilon}^{\varepsilon}\frac{\dd\theta}{\left|(1-e^{\ii\theta})^{1/2-\ii\xi}\right|}+2\pi C_{6}'\right)+[\xi\mapsto-\xi],
\end{align*}
for all $n\in\N$ and $\xi\in\mathcal{S}_{K,L}$. 
Using once more~\eqref{eq:complex_power_aux} and majorizing continuous functions on $\mathcal{S}_{K,L}$ by constants, we arrive at the upper bound
\begin{equation}
\left|\int_{-\varepsilon}^{\varepsilon}e^{-\ii(n-1)\theta}\left(f_{\xi}'(e^{\ii\theta})-g_{\xi}'(e^{\ii\theta})\right)\dd\theta\right|\leq C_{7}'\int_{0}^{\varepsilon}\frac{\dd\theta}{|1-e^{\ii\theta}|^{1/2+\Im\xi}}+C_{8}',
\label{eq:second_int_estim_inproof}
\end{equation}
for all $\xi\in\mathcal{S}_{K,L}$, $n\in\N$, and some $C_{7}',C_{8}'>0$.

It remains to verify that the integral from the right-hand side of~\eqref{eq:second_int_estim_inproof} is bounded. It follows from the elementary estimate
\begin{align*}
\int_{0}^{\varepsilon}\frac{\dd\theta}{|1-e^{\ii\theta}|^{1/2+\Im\xi}}&=\frac{1}{2^{1/2+\Im\xi}}\int_{0}^{\varepsilon}\frac{\dd\theta}{|\sin(\theta/2)|^{1/2+\Im\xi}}\leq \left(\frac{\pi}{2}\right)^{1/2+\Im\xi}\int_{0}^{\pi}\frac{\dd\theta}{\theta^{1/2+\Im\xi}}\\
&=\frac{\pi}{2^{\Im\xi-1/2}\,(1-2\Im\xi)}\leq\frac{2\pi}{1-2L}.
\end{align*}
Thus, we get
\begin{equation}
\left|\int_{-\varepsilon}^{\varepsilon}e^{-\ii(n-1)\theta}\left(f_{\xi}'(e^{\ii\theta})-g_{\xi}'(e^{\ii\theta})\right)\dd\theta\right|\leq C_{9}',
\label{eq:estim2_inproof}
\end{equation}
for all $\xi\in\mathcal{S}_{K,L}$ and $n\in\N$, where $C_{9}':=2C_{7}'\,\pi/(1-2L)+C_{8}'$. Finally, a combination of~\eqref{eq:estim1_inproof} and~\eqref{eq:estim2_inproof} implies the claim of Lemma~\ref{lem:remainder_estim_aux} for the constant $C_{K,L}':=C_{5}'+C_{9}'$.

\subsection{Higher order terms in the asymptotic formula for $\|L_{n}\|$}\label{subsec:a.2}

We continue the computational procedure started in the proof of Theorem~\ref{thm:l-hilbert_large_evls} for the particular case of the smallest root $\xi_{1,n}$. We make use of the asymptotic formula
\[
 \xi_{1,n}\sim\pi+\sum_{k=1}^{\infty}\frac{x_{k}}{\log^{k} n}, \quad n\to\infty,
\]
and the power series expansion
\[
 F(\xi)=\frac{\Gamma(1+2\ii\xi)}{\Gamma^{3}(1/2+\ii\xi)}=\sum_{l=0}^{\infty}\gamma_{l}\xi^{l}.
\]
It is follows from Proposition~\ref{prop:q_n_asympt} that 
\[
\lim_{n\to\infty}\log^{m}(n)\left[\Im F\!\left(\frac{\xi_{1,n}}{\log n}\right)\cos\xi_{1,n}+\Re F\!\left(\frac{\xi_{1,n}}{\log n}\right)\sin\xi_{1,n}\right]=0,
\]
for any power $m\in\N_{0}$ of the logarithm. Using these facts, we see, after simple manipulations, that the coefficients in the expression
\begin{align}
&-\sum_{l=0}^{\infty}\ii\gamma_{2l+1}\left(\frac{\pi}{\log n}+\sum_{k=1}^{\infty}\frac{x_{k}}{\log^{k+1} n}\right)^{2l+1}\times\sum_{m=0}^{\infty}\frac{(-1)^{m}}{(2m)!}\left(\sum_{k=1}^{\infty}\frac{x_{k}}{\log^{k} n}\right)^{2m}\nonumber\\
\hskip12pt &+\sum_{l=0}^{\infty}\gamma_{2l}\left(\frac{\pi}{\log n}+\sum_{k=1}^{\infty}\frac{x_{k}}{\log^{k+1} n}\right)^{2l}\times\sum_{m=0}^{\infty}\frac{(-1)^{m}}{(2m+1)!}\left(\sum_{k=1}^{\infty}\frac{x_{k}}{\log^{k} n}\right)^{2m+1}
\label{eq:init_eq_coeff}
\end{align}
standing in front of the same power of $\log n$ have to vanish. Then the unknown coefficients $x_{1},x_{2},x_{3},\dots$ can be successively computed from the resulting equations. 

By extracting the first three coefficients from~\eqref{eq:init_eq_coeff}, we get equations
\begin{align*}
x_{1}\gamma_{0}-\ii\pi\gamma_{1}&=0,\\
x_{2}\gamma_{0}-\ii x_{1}\gamma_{1}&=0,\\
(-x_{1}^{3}+6x_{3})\gamma_{0}+6\pi^{2}x_{1}\gamma_{2}+3\ii(\pi x_{1}^{2}\gamma_{1}-2x_{2}\gamma_{1}-2\pi^{3}\gamma_{3})&=0,
\end{align*}
which implies
\[
 x_{1}=\ii\frac{\pi\gamma_{1}}{\gamma_{0}}, \quad x_{2}=-\frac{\pi\gamma_{1}^{2}}{\gamma_{0}^{2}}, \quad
 x_{3}=\ii\frac{\pi(\pi^{2}-3)\gamma_{1}^{3}-3\pi^{3}\gamma_{0}\gamma_{1}\gamma_{2}+3\pi^{3}\gamma_{0}^{2}\gamma_{3}}{3\gamma_{0}^{3}}.
\]
Coefficients $\gamma_{0}=F(0)$ and $\gamma_{1}=F'(0)$ have been computed in~\eqref{eq:F_F'_id}. From the known values of the Gamma function and its derivatives at the points $1$ and $1/2$, see~\cite[\S~5.15]{dlmf}, we can also express other two coefficients as
\[
 \gamma_{2}=\frac{1}{2}F''(0)=\frac{5\pi^{2}-6\kappa^{2}}{12\pi^{3/2}}
 \quad\mbox{ and }\quad
 \gamma_{3}=\frac{1}{6}F'''(0)=\ii\frac{\kappa(5\pi^{2}-2\kappa^{2})-52\zeta(3)}{12\pi^{3/2}},
\]
where $\kappa:=\gamma+6\log 2$ and $\zeta(3)$ is Ap{\' e}ry's constant. Then we get
\[
 x_{1}=-\pi\kappa, \quad x_{2}=\pi\kappa^{2}, \quad x_{3}=-\pi\kappa^{3}+\frac{13}{3}\pi^{3}\zeta(3),
\]
which means
\[
 \xi_{1,n}=\pi-\frac{\pi\kappa}{\log n}+\frac{\pi\kappa^{2}}{\log^{2}n}-\frac{\pi(3\kappa^{3}-13\pi^{2}\zeta(3))}{3\log^{3}n}+O\left(\frac{1}{\log^{4}n}\right), \quad n\to\infty.
\]
Recalling finally that $\|L_{n}\|=\mu_{n,n}=4/(1+4\xi_{1,n}^{2}/\log^{2} n)$, see~\eqref{eq:mu_rel_xi}, one readily verifies asymptotic expansion~\eqref{eq:norm_asympt}.

\bibliographystyle{acm}

\end{document}